\documentclass{article}
\PassOptionsToPackage{numbers, compress}{natbib}
\usepackage[preprint]{neurips_2025}

\usepackage[utf8]{inputenc} 
\usepackage[T1]{fontenc}    
\usepackage{url}            
\usepackage{amssymb,amsfonts,amsmath,amsthm} 
\usepackage{enumerate,enumitem,graphicx,mathrsfs,eucal,verbatim, bbm, derivative}
\usepackage{tikz, tikz-cd}
\usepackage{svg}
\usepackage{wasysym}
\usepackage{rotating}
\usepackage{hyperref}     
\usepackage{subcaption}
\usepackage{booktabs}
\usepackage{stmaryrd}
\usepackage[normalem]{ulem}

\theoremstyle{plain}
\newtheorem{thm}{Theorem}[section]
\newtheorem{conj}[thm]{Conjecture}

\newtheorem{prop}[thm]{Proposition}
\newtheorem{lemm}[thm]{Lemma}
\newtheorem{cor}[thm]{Corollary}
\theoremstyle{definition}
\newtheorem{exmp}[thm]{Example}
\newtheorem{defn}{Definition}
\theoremstyle{remark}
\newtheorem{rmk}{Remark}[section]

\DeclareMathOperator{\rank}{rank}

\DeclareMathOperator{\ord}{ord}

\renewcommand{\AA}{\mathbb{A}}
\newcommand{\RR}{\mathbb{R}}

\newcommand{\KK}{\mathbb{K}}
\newcommand{\PP}{\mathbb{P}}
\newcommand{\ZZ}{\mathbb{Z}}

\newcommand{\Pcal}{\mathcal{P}}
\newcommand{\Ical}{\mathcal{I}}

\let\svthefootnote\thefootnote
\newcommand\freefootnote[1]{%
  \let\thefootnote\relax%
  \footnotetext{\hspace{-1.5em}#1}%
  \let\thefootnote\svthefootnote%
}

\title{Linear Independence of Polynomial Compositions \\ and Identifiability of Deep Neural Networks}

\author{
 Kathl\'en Kohn \\
 KTH Stockholm \\
 Digital Futures \\
 \texttt{kathlen@kth.se}
 \And
 Giovanni Luca Marchetti \\
 AI4S AB Stockholm \\
 \texttt{giovanni.marchetti@aiforscience.se} 
 \AND
 Alex Massarenti \\
 University of Ferrara \\
 \texttt{msslxa@unife.it} 
 \And
 Massimiliano Mella \\
 University of Ferrara \\
 \texttt{mll@unife.it}
}
\begin{document}

\maketitle

\begin{abstract}
Motivated by theoretical problems in deep learning, we conjecture that post-composing a fixed number of pairwise distinct nonconstant polynomials with a generic polynomial of sufficiently large degree yields linearly independent polynomials. This generalizes Newman--Slater's theorem on powers of polynomials. We establish several cases of this conjecture and its origin-passing variant: We prove the result for two polynomials, and for an arbitrary number of polynomials when their degrees are bounded. 
Furthermore, we show how the conjecture implies a complete understanding of the identifiability (i.e., parameter symmetries) of deep fully connected neural network architectures with generic polynomial activation functions. In particular, for network architectures with layer-specific activations of increasing degree, our established versions of the conjecture fully characterize the set of parameters yielding the same end-to-end network function. As a special case, we fully resolve the identifiability of shallow polynomial networks. 
\end{abstract}


\section{Introduction}\label{sec:intro}

One of the most fundamental questions for parametrized machine learning models is \emph{identifiability}: 
Given a function parametrized by the model, which parameter choices yield that function?
We investigate this question for the most classical deep learning architecture: \emph{multilayer perceptrons (MLPs)}. 
After fixing a so-called \emph{activation} $\sigma \colon \RR \to \RR$, such a model is parametrized by tuples of affine linear maps (of compatible dimensions): each tuple $(\alpha_1, \ldots, \alpha_L)$ yields a function by alternating the affine linear maps with the activation as $\alpha_L \circ \sigma \circ \cdots \circ \sigma \circ \alpha_2 \circ \sigma \circ \alpha_1$, where $\sigma$ is applied entrywise to vectors. 
We focus on MLPs without bias vectors, which means that each $\alpha_i$ is a linear map.

The identifiability question for MLPs has been answered for various choices of activations, e.g., for sigmoidal activations \cite{fefferman1994reconstructing,vlavcic2022neural} or hyperbolic tangent \cite{sussmann1992uniqueness,vlavcic2021affine}. For ReLU activations, identifiability is considered to be a challenging problem with only partial results available \cite{petzka2020notes,phuong2020functional,grigsby2023hidden,grillo2026most,gegenfurtner2026symmetries}.

The recently established subfield of deep learning theory deemed \emph{neuroalgebraic geometry} \cite{marchetti2025invitationneuroalgebraicgeometry} proposes to investigate \emph{polynomial} machine learning models. 
This is motivated by the fact that, on the one hand, the study of polynomial models is tractable via the rich toolbox of algebraic geometry and, on the other hand, the approximation capabilities of polynomials suggest that polynomial models can reflect the generally expected behavior of machine learning models.
For MLPs with polynomial activation $\sigma$, the identifiability question is (essentially) resolved when $\sigma$ is either linear (see, e.g., \cite{trager2019pure}) or a monomial of large degree \cite{finkel2024activation,cruaciun2025linear,usevich2025identifiability,massarenti2025alexander}.
For the agenda of neuroalgebraic geometry to use polynomial models as approximators of arbitrary models, it would be most interesting to establish identifiability results when $\sigma$ is a generic polynomial of large degree. 
However, the only known result in this setting is that, for a generic function parametrized by the model, there are finitely many parameter choices \cite[Theorem 4.1]{shahverdi2026learning}. 
Resolving the identifiability question would entail determining a precise condition for the term ``generic'' in the previous statement and identifying exactly which parameter choices occur in those finite fibers.
We address this problem in this article.

We wish to point out that the identifiability of other (i.e., non-MLP) polynomial machine learning models is known, namely for convolutional neural networks with a monomial activation of large degree \cite{shahverdi2024geometryoptimizationpolynomialconvolutional} or a generic polynomial activation \cite{shahverdi2026learning} or for single-head attention networks \cite{henry2024geometrylightningselfattentionidentifiability}.

For MLPs with a generic polynomial activation of large degree, the identifiability question leads us to an open problem in commutative algebra.
To motivate that problem, we recall the following known result due to Newman and Slater, which holds over any field $\KK$ of characteristic zero:
\begin{thm}[\cite{newman1979waring}]
\label{thm:newman-slater}
    Let $p_1, \ldots, p_k$, $k \geq 2$, be multivariate polynomials that are nonzero and pairwise nonproportional. If $r > k(k-2)$, then $p_1^r, \ldots, p_k^r$ are linearly independent. 
\end{thm}
For $k=3$, that result can be viewed as a polynomial version of Fermat's Last Theorem. 
The Newman--Slater result was the main ingredient in resolving identifiability for monomial MLPs \cite{cruaciun2025linear,finkel2024activation}.
Intuitively, it says that plugging polynomials into a monomial activation of large degree separates the polynomials enough so that they become linearly independent. 
We conjecture that not only monomial activations behave like that, but almost all polynomial activations:

\begin{conj} \label{conj:strong}
    Let $k \geq 2$. There is an integer $R(k)$ such that, for every $r \geq R(k)$, there is a nonempty Zariski open subset $U_{k,r} \subseteq \KK[z]_{\leq r}$ with the following property:
    For every $\sigma \in U_{k,r}$ and all multivariate polynomials $p_1, \ldots, p_k$ that are nonconstant and pairwise distinct, the polynomials $\sigma(p_1), \ldots, \sigma(p_k)$ are linearly independent. 
\end{conj}

For our purposes of investigating polynomial MLPs, we need Conjecture~\ref{conj:strong} in its \emph{origin-passing variant}, which is the verbatim statement as above but only considering polynomials $\sigma$ and $p_j$ that satisfy $\sigma(0)=0$ and $p_j(0)=0$ for all $j = 1,\ldots, k$ (see Conjecture~\ref{conj:goat_swapped}).
Proving the origin-passing variant of the conjecture would resolve the identifiability question for MLPs with generic polynomial activations of large degree, as we explain in Section~\ref{sec:MLP}.
In fact, MLP identifiability already follows from the weaker Conjecture~\ref{conj:weakOrigin}.

We prove two restricted versions of Conjecture~\ref{conj:strong} and its origin-passing variant:
First, we provide proofs for the case $k=2$ (Propositions~\ref{prop:twoPols} and~\ref{prop:twoPolsOrigin}), which resolves identifiability for deep MLPs of width one (see Section~\ref{sec:MLP}).
Second, we prove Conjecture~\ref{conj:strong} and its origin-passing variant when the bound $R$ depends not only on the number $k$ of polynomials but also on their maximum degree (Theorems~\ref{thm:main-nonmonic} and \ref{thm:variable-degree-main} for polynomials with constant term and Theorems \ref{thm:equalDegreeOrigin} and \ref{thm:boundedDegreeOrigin} for origin-passing polynomials).
Note also that the original Newman--Slater result is independent of the degrees of the polynomials $p_1, \ldots, p_k$.
Our degree-dependent result solves the identifiability problem for shallow MLPs and for deep MLPs where the degrees of the generic polynomial activations grow from layer to layer (see Section~\ref{sec:MLP}).

Furthermore, Examples~\ref{ex:sigma_k_ex1}, \ref{ex:actOrign1}, and \ref{ex:actOrign2} give explicit families of fixed polynomial activations $\sigma$ that turn every collection of nonconstant, pairwise distinct (and origin-passing) polynomials $p_1,\ldots,p_k$ into linearly independent $\sigma(p_1),\ldots,\sigma(p_k)$. 
For a deep MLP with an origin-passing activation from such a family, identifiability holds without any growing-degree assumptions.

\section{Linear independence of polynomial compositions} 
\label{sec:general}
Let $\KK$ be a field of characteristic zero, $k\in\ZZ_{\geq 2}$, and $r\in\ZZ_{\geq 0}$.
Given a polynomial $\sigma\in \KK[z]_{\leq r}$ and polynomials $p_1,\dots,p_k\in \KK[x_1,\ldots,x_n]$, we are interested in linear relations among the compositions $\sigma(p_1),\dots,\sigma(p_k)$. 
We first observe that, for all conjectures and results proved in this article, it is sufficient to consider univariate polynomials.

\subsection{Reduction to one variable}
Let $p_1,\dots,p_k$ be pairwise distinct polynomials in $\KK[x_1,\dots,x_n]$, and assume that there is a nontrivial linear combination 
$$
\sum_{i=1}^k c_i\sigma(p_i(x_1,\dots,x_n))=0
$$
in $\KK[x_1,\dots,x_n]$. 
Composing with a general linear map $\ell:\AA^1\to\AA^n$ preserves the degrees of the $p_i$, keeps them pairwise distinct, and gives the relation
$$
\sum_{i=1}^k c_i\sigma(p_i(\ell(t)))=0
$$
in $\KK[t]$. 
Moreover, if the $p_i$ are origin-passing (i.e., $p_i(0)=0$), then so is the composition $p_i \circ \ell$.
Henceforth, we assume $p_1, \ldots, p_k \in \KK[t]$.

\subsection{Two polynomials}
The special case of two polynomials (i.e., $k=2$) can be easily settled. We now show that $R(2) = 3$ in Conjecture~\ref{conj:strong} and provide an explicit open set of activation polynomials.

\begin{prop}
\label{prop:twoPols}
Let $r\geq 3$ and write $\sigma(z)=a_0+a_1z+\cdots+a_rz^r$. Set $\Delta_2(\sigma)=2r a_ra_{r-2}-(r-1)a_{r-1}^2$ and
$$
\Delta_3(\sigma)=3r^2a_r^2a_{r-3}-3r(r-2)a_ra_{r-1}a_{r-2}+(r-1)(r-2)a_{r-1}^3.
$$
Then $U_r=\{\sigma\in\KK[z]_{\leq r}:a_r\Delta_2(\sigma)\Delta_3(\sigma)\neq 0\}$ is a nonempty Zariski open subset of $\KK[z]_{\leq r}$ such that, for every $\sigma\in U_r$ and every pair of distinct nonconstant polynomials $p,q\in\KK[t]$, the polynomials $\sigma(p)$ and $\sigma(q)$ are linearly independent. 
\end{prop}

\begin{proof}
The set $U_r$ is Zariski open and nonempty since $z^r+z^{r-2}+z^{r-3}\in U_r$. Let $\sigma\in U_r$ and assume that $s\sigma(p)=\sigma(q)$ for some nonconstant $p,q\in\KK[t]$ and $s\in\KK^*$. Comparing degrees gives $\deg p=\deg q=: \delta$. If $b\in\KK^*$ is the ratio of the leading coefficient of $q$ to the leading coefficient of $p$, then comparison of leading coefficients in the equality $s\sigma(p)=\sigma(q)$ gives $s=b^r$.

Set $c:=q-bp$. If $c$ is nonconstant, then $\deg c<\delta$ and $q^r-b^rp^r$ has degree $(r-1)\delta+\deg c$, while every term of $\sigma(q)-b^r\sigma(p)$ coming from powers smaller than $r$ has degree at most $(r-1)\delta$. This contradicts $\sigma(q)=b^r\sigma(p)$. Hence, $c\in\KK$.

Since $p$ is nonconstant, the substitution map $\KK[z]\to\KK[t]$, $f\mapsto f(p)$, is injective.Therefore, $ \sigma(bp+c)= \sigma(q)=b^r\sigma(p)$ implies $\sigma(bz+c)=b^r\sigma(z)$. Comparing the coefficients of $z^{r-1}$ gives $c=(b-1)a_{r-1}/(ra_r)$. Set $\alpha:=-a_{r-1}/(ra_r)$ and $\widetilde\sigma(z):=\sigma(z+\alpha)$. Then, $c=(1-b)\alpha$ and $\widetilde\sigma(bz)=b^r\widetilde\sigma(z)$. The coefficient of $z^{r-1}$ in $\widetilde\sigma$ is zero, while those of $z^{r-2}$ and $z^{r-3}$ are $\Delta_2(\sigma)/(2ra_r)$ and $\Delta_3(\sigma)/(3r^2a_r^2)$, respectively. They are nonzero since $\sigma\in U_r$. 
Comparing the coefficients of $z^{r-2}$ and $z^{r-3}$ in $\widetilde\sigma(bz)=b^r\widetilde\sigma(z)$ gives $b^{r-2}=b^r$ and $b^{r-3}=b^r$, i.e., $b^2=b^3=1$. Thus, $b=1$, which implies that $c=0$ and $p=q$.
\end{proof}

\subsection{Polynomials of equal degree}

Next, we prove a version of Conjecture~\ref{conj:strong} where the lower bound $R$ depends not only on the number $k$ of polynomials but also on their degree.
We begin with the case where the polynomials $p_1,\ldots,p_k$ have the same degree. This is the natural first setting for the problem since the coefficients of degree $m$ polynomials form an affine space and the condition that $\sigma(p_1),\ldots,\sigma(p_k)$ are linearly dependent is closed in the corresponding parameter space. We will prove that, after fixing $k$ and $m$, a general polynomial $\sigma$ of sufficiently large degree separates every ordered $k$-tuple of pairwise distinct degree $m$ polynomials in the sense that the compositions $\sigma(p_1),\ldots,\sigma(p_k)$ are linearly independent.

\begin{lemm}\label{lem:filtered-rank-bound}
Let $p_1,\dots,p_k\in \KK[t]$ be pairwise distinct polynomials of the same degree $m\geq 1$ and let $(\lambda_1,\cdots,\lambda_k)\in\KK^{k} \setminus \{ (0,\ldots,0) \}$. Consider the $\KK$-linear map
\begin{equation}\label{eq:T-same-degree}
T_{p,\lambda}:\KK[z]_{\leq r}\longrightarrow \KK[t]_{\leq rm},\qquad
\sigma\longmapsto \sum_{i=1}^k \lambda_i\sigma(p_i(t)).
\end{equation}
Then, $\rank T_{p,\lambda}\geq \lfloor(r+1)/k\rfloor$ and
$$
\dim\ker T_{p,\lambda}\leq (r+1)-\left\lfloor\frac{r+1}{k}\right\rfloor.
$$
\end{lemm}

\begin{proof}
Set $q=t^{-1}$. Since each $p_i$ has degree $m$, we may write $p_i(t)=t^mu_i(q)$, where $u_i(q)=\alpha_i+u_{i,1}q+\cdots+u_{i,m}q^m\in \KK[q]$ and $\alpha_i\in \KK^*$. Moreover, the power series $u_1(q),\dots,u_k(q)$ are pairwise distinct.

For every $h\geq 0$, set $Q_h(t):=\sum_i\lambda_i p_i(t)^h$. Then, $Q_h(t)=t^{mh}F_h(q)$, where $F_h(q)=\sum_i\lambda_i u_i(q)^h$. Write $F_h(q)=\sum_{\nu\geq 0}a_\nu(h)q^\nu$. If all sequences $a_\nu(h)$ were zero, then $F_h(q)=0$ for every $h\geq 0$. In particular, $F_0(q)=\cdots=F_{k-1}(q)=0$.
The latter system of equations can be written as 
$V(q)(\lambda_1,\dots,\lambda_k)^\top=0$, where 
$$
V(q):=\bigl(u_i(q)^h\bigr)_{0\leq h\leq k-1,\ 1\leq i\leq k}.
$$
This Vandermonde matrix has determinant $\prod_{i<j}(u_j(q)-u_i(q))\neq 0$ in $\KK[[q]]$. Hence, $V(q)(\lambda_1,\dots,\lambda_k)^\top=0$ forces $\lambda_1=\cdots=\lambda_k=0$, a contradiction.

Let $\nu_0$ be the smallest integer such that $a_{\nu_0}(h)$ is not the zero sequence. The sequence $F_h(q)$ is annihilated by $\prod_i(E-u_i(q))$, where $E$ is the shift operator in $h$. Thus, there are $c_0(q),\dots,c_k(q)\in \KK[[q]]$, with $c_k(q)=1$, such that
\begin{equation}\label{eq:recurrence-same-degree-q}
\sum_{j=0}^k c_j(q)F_{h+j}(q)=0
\end{equation}
for every $h\geq 0$. Looking at the coefficient of $q^{\nu_0}$ in \eqref{eq:recurrence-same-degree-q}, we see that all lower order terms vanish by the minimality of $\nu_0$ and, setting $a(h) := a_{\nu_0}(h)$, that
\begin{equation}\label{eq:recurrence-same-degree}
\sum_{j=0}^k c_j(0)a(h+j)=0.
\end{equation}
The coefficients $c_j(0)$ are those of $\prod_i(X-\alpha_i)$. In particular, $c_0(0)$ and $c_k(0)$
are nonzero. 
Therefore, the nonzero sequence $a(h)$ cannot have $k$ consecutive zero values.
(Indeed, if $a(n)=a(n+1)=\ldots=a(n+k-1)$ for some $n$, applying \eqref{eq:recurrence-same-degree} for $h=n$ yields $c_k(0)a(n+k)=0$, i.e., $a(n+k)=0$. Then, applying \eqref{eq:recurrence-same-degree} for $h=n+1, n+2, \ldots$, shows that $a(h)=0$ for all $h \geq n$. Similarly, applying \eqref{eq:recurrence-same-degree} for $h=n-1$ gives $c_0(0) a(n-1)=0$, i.e., $a(n-1)=0$. Continuing this for $h=n-2, n-3, \ldots$ shows that the whole sequence $a(h)$ must be zero, a contradiction).

It follows that in every block of $k$ consecutive integers, there is an $h$ with $a_{\nu_0}(h)\neq 0$. Hence, among $0,\dots,r$, there are at least $\lfloor(r+1)/k\rfloor$ such integers. For each of them, $\deg Q_h=mh-\nu_0$, and these degrees are distinct. Thus, the corresponding $Q_h=T_{p,\lambda}(z^h)$ are linearly independent. This gives the rank estimate, and the kernel estimate follows from $\dim \KK[z]_{\leq r}=r+1$.
\end{proof}

\begin{thm}\label{thm:main-nonmonic}
Let $m,k\geq 1$ and assume 
$$
r\geq k^2(m+2)-1.
$$ 
Then, there exists a nonempty Zariski open subset $U_{k,m,r}\subseteq\KK[z]_{\leq r}$ such that, for every $\sigma\in U_{k,m,r}$ and every $k$-tuple of pairwise distinct polynomials $p_1,\dots,p_k\in \KK[t]$ of degree $m$, the polynomials $\sigma(p_1),\dots,\sigma(p_k)$ are linearly independent in $\KK[t]$.
\end{thm}

\begin{proof}
Let $\Pcal_m\subseteq\KK[t]_{\leq m}$ be the open subset of polynomials of degree precisely $m$.
Furthermore, let $\Pcal_{k,m}\subseteq\Pcal_m^k$ be the open subset of pairwise distinct $k$-tuples of such polynomials. Note that $\dim\Pcal_{k,m}=k(m+1)$.

Consider the incidence variety
\begin{align} \label{eq:incidenceVar}
    \Ical_{k,m,r}:=\left\{([\sigma],p_1,\dots,p_k,[\lambda])\in\PP(\KK[z]_{\leq r})\times\Pcal_{k,m}\times\PP^{k-1}\ :\ \sum_{i=1}^k\lambda_i\sigma(p_i(t))=0\right\}
\end{align}
and the projection $\pi_{2,3}:\Ical_{k,m,r}\to\Pcal_{k,m}\times\PP^{k-1}$. For fixed $(p,\lambda)$, the fiber is $\PP(\ker T_{p,\lambda})$, where $T_{p,\lambda}$ is the map in \eqref{eq:T-same-degree}. Set $L:=\lfloor(r+1)/k\rfloor$. By Lemma~\ref{lem:filtered-rank-bound}, every nonempty fiber has dimension at most $r-L$. Hence,
\begin{equation}\label{eq:incidence-nonmonic-bound}
\dim\Ical_{k,m,r}\leq k(m+1)+(k-1)+r-L.
\end{equation}
Since $r\geq k^2(m+2)-1$, we have $L\geq k(m+2)$. Substituting this in \eqref{eq:incidence-nonmonic-bound}, we get $\dim\Ical_{k,m,r}\leq r-1$.

Now, consider the projection $\pi_1:\Ical_{k,m,r}\to\PP(\KK[z]_{\leq r})$. Its image variety $C:=\overline{\pi_1(\Ical_{k,m,r})}$  is a proper closed subset of $\PP(\KK[z]_{\leq r}) \cong \PP^r$. If $\widehat C\subseteq\KK[z]_{\leq r}$ denotes the affine cone over $C$, then $U_{k,m,r}=\KK[z]_{\leq r}\setminus\widehat C$ is nonempty and open. By construction, a point $\sigma\in U_{k,m,r}$ cannot yield a nontrivial relation among $\sigma(p_1),\dots,\sigma(p_k)$ for any pairwise distinct $p_i$ of degree $m$.
\end{proof}

\subsection{Polynomials of bounded degree}

We now pass from polynomials of the same fixed degree to polynomials of bounded degree. 
For fixed $k$ and $m$, our goal is to show that, for general $\sigma$ of sufficiently large degree, the compositions $\sigma(p_1),\ldots,\sigma(p_k)$ are linearly independent for every $k$-tuple of pairwise distinct nonconstant polynomials $p_i$ of degree at most $m$.
By treating simultaneously all possible degenerations in which some leading coefficients of the $p_i$ vanish, the same incidence construction as above gives a closed bad locus of $\sigma$.

\begin{lemm}\label{lem:variable-degree-rank-bound}
Let $p_1,\dots,p_k\in \KK[t]$ be pairwise distinct nonconstant polynomials of degree at most~$m$ and let $(\lambda_1,\cdots,\lambda_k)\in\KK^{k} \setminus \{  (0,\ldots,0) \}$. Consider the $\KK$-linear map
\begin{equation}\label{eq:T-variable-degree}
T_{p,\lambda}:\KK[z]_{\leq r}\longrightarrow \KK[t]_{\leq rm},\qquad
\sigma\longmapsto \sum_{i=1}^k \lambda_i\sigma(p_i(t)).
\end{equation}
Then,
\begin{equation}\label{eq:rank-variable-degree-bound}
\rank T_{p,\lambda}\geq
\left\lfloor\frac{r}{k}-\frac{m(k-1)}{2}\right\rfloor.
\end{equation}
Consequently,
\begin{equation}\label{eq:kernel-variable-degree-bound}
\dim\ker T_{p,\lambda}\leq
(r+1)-\left\lfloor\frac{r}{k}-\frac{m(k-1)}{2}\right\rfloor.
\end{equation}
\end{lemm}

\begin{proof}
Let $S\!:=\!\{i:\lambda_i\neq 0\}$ and $D\!:=\!\max_{i\in S}\deg p_i\geq 1$. Further, let \mbox{$I:=\{i\in S:\deg p_i=D\}$} and set $s:=|I| \geq 1$. For $i\in I$, write $p_i(t)=t^D u_i(q)$, where $q:=t^{-1}$ and $u_i(q)\in \KK[q]$ with $u_i(0)\neq 0$. Note that the $u_i(q)$, for $i\in I$, are pairwise distinct.

For every $h\geq 0$, consider
$$
Q_h^{\rm top}(t):=\sum_{i\in I}\lambda_i p_i(t)^h=t^{Dh}F_h(q),\quad
\text{ where }
F_h(q):=\sum_{i\in I}\lambda_i u_i(q)^h.
$$
Write $F_h(q)=\sum_{\nu\geq 0}a_\nu(h)q^\nu$. As in the proof of Lemma~\ref{lem:filtered-rank-bound}, there is a smallest $\nu_0$ such that $a_{\nu_0}(h)$ is not the zero sequence.

Let
$$
D_I:=\ord_q\left(\prod_{i<j,\ i,j\in I}(u_j(q)-u_i(q))\right),
$$
with $D_I=0$ if $s=1$. We claim that $\nu_0\leq D_I$. If $\nu_0>D_I$, then $F_h(q)$ is zero modulo $q^{D_I+1}$ for all $h \geq 0$. 
In particular, for the Vandermonde matrix 
$$
V(q):=\bigl(u_i(q)^h\bigr)_{0\leq h\leq s-1,\ i \in I}
$$
and the vector $\lambda_I := (\lambda_i)_{i \in I}$, we obtain
$V(q)\lambda_I^\top=0$ modulo $q^{D_I+1}$.
Multiplying this system by the adjugate matrix of $V(q)$ gives that $\det V(q)\lambda_I=0$ modulo $q^{D_I+1}$. 
Since $\lambda_i \neq 0$ for all $i \in I$, we conclude that $\det V(q)$ is divisible by $q^{D_I+1}$.
However, 
$\ord_q\det V(q)=D_I$, a contradiction.Therefore,
\begin{equation}\label{eq:DI-bound}
\nu_0\leq D_I\leq D\binom{s}{2}\leq m\binom{k}{2}.
\end{equation}

As in the proof of Lemma~\ref{lem:filtered-rank-bound}, the sequence $a_{\nu_0}(h)$ satisfies a recurrence
$\sum_{j=0}^s c_j a_{\nu_0}(h+j)=0$
 with $c_0 \neq 0$ and $c_s \neq 0$, obtained from $\prod_{i\in I}(E-u_i(q))$ by taking the coefficient of $q^{\nu_0}$. Thus, it cannot have $s$ consecutive zero values. 
Hence, among $\nu_0+1,\ldots,r$, there are at least $\lfloor(r-\nu_0)/s\rfloor$ values of $h$ with $a_{\nu_0}(h)\neq0$. Moreover, \eqref{eq:DI-bound} gives
$$
\frac{r-\nu_0}{s}\geq \frac{r}{s}-\frac{D(s-1)}{2}
\geq \frac{r}{k}-\frac{m(k-1)}{2}.
$$
Therefore, there are at least as many values of $h$ as the quantity in \eqref{eq:rank-variable-degree-bound} with $a_{\nu_0}(h)\neq0$ and $\nu_0 < h \leq r$.

For $h$ with $a_{\nu_0}(h) \neq 0$, we have that $\deg Q_h^{\rm top}=Dh-\nu_0$. 
Moreover, $\deg (\sum_{i\in S \setminus I}\lambda_i p_i(t)^h)\leq (D-1)h$.
Whenever $h>\nu_0$, the part of degree $Dh-\nu_0$ cannot be canceled and so
$T_{p,\lambda}(z^h) = \sum_{i\in S }\lambda_i p_i(t)^h$ has degree $Dh-\nu_0$.
Therefore, those polynomials $T_{p,\lambda}(z^h)$ have distinct degrees and are linearly independent. This proves \eqref{eq:rank-variable-degree-bound}, and \eqref{eq:kernel-variable-degree-bound} follows from $\dim \KK[z]_{\leq r}=r+1$.
\end{proof}

\begin{thm}\label{thm:variable-degree-main}
Let $m,k\geq 1$ and assume
\begin{equation}\label{eq:main-bound-variable-degree}
r \geq k\left(k(m+1)+k+\frac{m(k-1)}{2}\right)
=\frac{k}{2}(3km+4k-m).
\end{equation}
Then, there exists a nonempty Zariski open subset $U_{k,m,r}\subseteq\KK[z]_{\leq r}$ such that, for every $\sigma\in U_{k,m,r}$ and every $k$-tuple of pairwise distinct nonconstant polynomials $p_1,\dots,p_k\in \KK[t]$ of degree at most~$m$, the polynomials $\sigma(p_1),\dots,\sigma(p_k)$ are linearly independent in $\KK[t]$.
\end{thm}

\begin{proof}
Let $\Pcal_{\leq m}^+\subseteq\KK[t]_{\leq m}$ be the open subset of nonconstant polynomials. Moreover, let $\Pcal_{k,\leq m}^+\subseteq(\Pcal_{\leq m}^+)^k$ be the open subset of pairwise distinct $k$-tuples of such polynomials. Note that $\dim\Pcal_{k,\leq m}^+=k(m+1)$.

Consider the incidence variety
\begin{align}\label{eq:incidenceBounded}
    \Ical_{k,m,r}:=\left\{([\sigma],p_1,\dots,p_k,[\lambda])\in\PP(\KK[z]_{\leq r})\times\Pcal_{k,\leq m}^+\times\PP^{k-1}\ :\ \sum_{i=1}^k\lambda_i\sigma(p_i(t))=0\right\}.
\end{align}
The fiber over $(p,\lambda)$ under the projection $\pi_{2,3}:\Ical_{k,m,r}\to\Pcal_{k,\leq m}^+\times\PP^{k-1}$ is $\PP(\ker T_{p,\lambda})$, where $T_{p,\lambda}$ is the map in \eqref{eq:T-variable-degree}. Set $L:=\left\lfloor r/k-m(k-1)/2\right\rfloor$. By Lemma~\ref{lem:variable-degree-rank-bound}, every nonempty fiber has dimension at most $r-L$. Hence,
\begin{equation}\label{eq:variable-degree-incidence-dim-bound}
\dim\Ical_{k,m,r}\leq k(m+1)+(k-1)+r-L.
\end{equation}
The assumption \eqref{eq:main-bound-variable-degree} gives $L\geq k(m+1)+k$. Substituting this in \eqref{eq:variable-degree-incidence-dim-bound}, we get $\dim\Ical_{k,m,r}\leq r-1$.

Now, let $C:=\overline{\pi_1(\Ical_{k,m,r})}$, where $\pi_1: \Ical_{k,m,r} \to \PP(\KK[z]_{\leq r})$ is the projection onto the first factor. Then, $C$ is a proper closed subset of $\PP(\KK[z]_{\leq r}) \cong \PP^r$. If $\widehat C\subset\AA^{r+1}$ denotes its affine cone, the open set $U_{k,m,r}=\AA^{r+1}\setminus\widehat C$ has the desired property by definition of the incidence variety.
\end{proof}

We wish to point out that it would also be interesting to prove the following weaker version of Conjecture~\ref{conj:strong}, where the bound $R$ only depends on the number $k$ of polynomials, but the Zariski open subset $U$ is allowed to depend on their maximum degree $m$:

\begin{conj}[weak]\label{conj:weak}
    Let $k \geq 2$. There is an integer $R(k)$ such that, for every $r \geq R(k)$ and every $m \geq 1$, there is a nonempty Zariski open subset $U_{k,r,m} \subseteq \KK[z]_{\leq r}$ with the following property:
    For every $\sigma \in U_{k,r,m}$ and all    $p_1, \ldots, p_k \in \KK[t]$ of degree at most $m$ that are nonconstant and pairwise distinct, the polynomials $\sigma(p_1), \ldots, \sigma(p_k)$ are linearly independent. 
\end{conj}

In fact, the origin-passing version of this weaker conjecture (see Conj. \ref{conj:weakOrigin}) is sufficient to resolve the problem of identifiability for polynomial MLPs, as we will explain in Section~\ref{sec:MLP}.

\subsection{Examples for $\sigma$}

We close this section with an explicit family of polynomials showing that the difficulty in Conjecture~\ref{conj:strong} is the openness statement rather than the existence of a single activation that works uniformly for all $p_i$.

\begin{exmp}
\label{ex:sigma_k_ex1}
Let $k\geq2$, take any $\ell > k(k-2)$, and let $f\in U_d$ be any polynomial from Proposition~\ref{prop:twoPols} for some $d\geq3$. Then $\sigma_{\ell,f}(z):=f(z)^{\ell}$ has the following property: For every choice of pairwise distinct nonconstant polynomials $p_1,\ldots,p_k$, the polynomials $\sigma_{\ell,f}(p_1),\ldots,\sigma_{\ell,f}(p_k)$ are linearly independent. Indeed, Proposition~\ref{prop:twoPols} implies that $f(p_1),\ldots,f(p_k)$ are pairwise nonproportional, and so (due to $\ell>k(k-2)$) the Newman--Slater theorem~\ref{thm:newman-slater} yields the claim.

For instance, both $f(z)=z^3+z+1$ and $f(z)=z^3+z^2+z$ belong to $U_3$: the corresponding values of $(\Delta_2,\Delta_3)$ are $(6,27)$ and $(4,-7)$, respectively. Thus, $(z^3+z+1)^{k(k-2)+1}$ and $(z^3+z^2+z)^{k(k-2)+1}$ are two explicit choices. Note that the second one is origin-passing.
\end{exmp}

In particular, the polynomial $\sigma_{\ell,f}$ belongs to $\KK[z]_{\leq r}$ for every $r\geq d \ell$. Thus, Example~\ref{ex:sigma_k_ex1} proves the existence part of Conjecture~\ref{conj:strong} after dropping the requirement that the good set be Zariski open.

\section{Origin-Passing Polynomials}
\label{sec:originPassing}

We now focus on the origin-passing analog of Conjecture~\ref{conj:strong} and its weaker version Conjecture~\ref{conj:weak}.
To this end, we denote by 
$\KK[t]^0:=\{p \in \KK[t]:p(0)=0\}$
and $\KK[t]^0_{\leq r}:=\KK[t]^0 \cap \KK[t]_{\leq r}$ the $\KK$-vector spaces of origin-passing polynomials (of degree at most $r$).

\begin{conj}[origin-passing]
\label{conj:goat_swapped}
Let $k \geq 2$. There is an integer $R(k)$ such that, for every  $r \geq R(k)$, there exists a nonempty Zariski open subset $U_{k,r} \subseteq \KK[z]^0_{\leq r}$ with the following property:
For every $\sigma \in U_{k,r}$ 
and every collection of nonzero, pairwise distinct
$p_1,\dots,p_k \in \KK[t]^0$, the polynomials 
$\sigma(p_1),\dots,\sigma(p_k)$ are linearly independent.
\end{conj}

\begin{conj}[weak, origin-passing]\label{conj:weakOrigin}
    Let $k \geq 2$. There is an integer $R(k)$ such that, for every $r \geq R(k)$ and every $m \geq 1$, there is a nonempty Zariski open subset $U_{k,r,m} \subseteq \KK[z]_{\leq r}^0$ with the following property:
    For every $\sigma \in U_{k,r,m}$ and all    $p_1, \ldots, p_k \in \KK[t]^0_{\leq m}$ that are nonzero and pairwise distinct, the polynomials $\sigma(p_1), \ldots, \sigma(p_k)$ are linearly independent. 
\end{conj}

We now prove analogues of the preceding results under the origin-passing assumption, following the same structure as in Section~\ref{sec:general}.

\subsection{Two polynomials}
\begin{prop}
\label{prop:twoPolsOrigin}
Suppose that $r \geq 2$, and let $\sigma(z)=a_1z+a_2z^2+\cdots+a_rz^r\in\KK[z]^0_{\leq r}$ be such that $a_r,a_{r-1}\neq 0$. Then, for every pair of nonzero distinct $p, q \in \KK[t]^0$, the polynomials
$\sigma(p)$ and $\sigma(q)$ are non-proportional. 
\end{prop}

\begin{proof}
Suppose that $s\sigma(p)=\sigma(q)$ for two nonzero polynomials $p,q \in  \KK[t]^0$ and $s \in \KK^*$. Note that $p$ and $q$ must have equal degree, which we denote by $\delta$. Let $b\in \KK^*$ be the ratio of the highest-degree coefficient of $q$ to that of $p$. By comparing the highest-degree coefficients of $s \sigma(p)$ and $\sigma(q)$, we deduce that $s=b^r$.

We wish to show that $q=bp$. To this end, define $c:=q-bp\in\KK[t]_{\leq \delta}^0$, and assume for contradiction that $\deg c \geq 1$. Then, $q^r-(bp)^r
$ has degree $(r-1)\delta + \deg c$, while any other term in $\sigma(q) - s \sigma(p)$ coming from lower powers $p^i, q^i$, with $i < r$, has degree at most $(r-1)\delta$.Therefore, the term $q^r-(bp)^r$ cannot be canceled out, leading to a contradiction with the equality $s\sigma(p)=\sigma(q)$. 

Thus, $c$ is a constant, and in fact must be $0$ since $p$ and $q$ are origin-passing.
We have shown that $s\sigma(p)=\sigma(bp)$. Since $p$ is nonconstant, the substitution map $\KK[z]\to\KK[t]$, $f\mapsto f(p)$, is injective. Hence, $s\sigma(z)=\sigma(bz)$. By comparing the terms of degree $r-1$, we conclude that $s=b^r=b^{r-1}$, implying that $b=1$, i.e., $q = p$.  
\end{proof}

\subsection{Polynomials of equal degree}

\begin{thm}\label{thm:equalDegreeOrigin}
Let $m,k\geq 1$ and assume $r\geq k^2(m+1)$. Then there exists a nonempty Zariski open subset $U_{k,m,r}\subseteq\KK[z]^0_{\leq r}$ such that, for every $\sigma\in U_{k,m,r}$ and every $k$-tuple of pairwise distinct polynomials $p_1,\ldots,p_k\in\KK[t]^0$ of degree $m$, the polynomials $\sigma(p_1),\ldots,\sigma(p_k)$ are linearly independent.
\end{thm}

\begin{proof}
Let $\Pcal_m^0\subseteq\KK[t]_{\leq m}^0$ be the open subset of origin-passing polynomials of degree equal to $m$, and let $\Pcal_{k,m}^0\subseteq(\Pcal_m^0)^k$ be the set of pairwise distinct $k$-tuples. Note that $\dim\Pcal_{k,m}^0=km$. For $p\in\Pcal_{k,m}^0$ and $\lambda\in\KK^{k} \setminus \{ 0 \}$, restrict the map in \eqref{eq:T-same-degree} to $T^0_{p,\lambda}:\KK[z]^0_{\leq r}\to\KK[t]^0_{\leq rm}$. The proof of Lemma~\ref{lem:filtered-rank-bound}, applied to $h=1,\ldots,r$, gives $\rank T^0_{p,\lambda}\geq L:=\lfloor r/k\rfloor$: the sequence $a_{\nu_0}(h)$ used there cannot vanish at $k$ consecutive values, and the corresponding $T^0_{p,\lambda}(z^h)$ have distinct degrees.

We define the incidence variety $\Ical^0_{k,m,r}$ as in \eqref{eq:incidenceVar}, 
with $\Pcal_{k,m}$ replaced by $\Pcal_{k,m}^0$ and $\PP(\KK[z]_{\leq r})$ by $\PP(\KK[z]^0_{\leq r})\cong\PP^{r-1}$. 
The same argument as in the proof of Theorem~\ref{thm:main-nonmonic} gives
$$
\dim\Ical^0_{k,m,r}\leq km+(k-1)+(r-L-1)=r+km+k-L-2.
$$
Since $r\geq k^2(m+1)$, we have $L\geq k(m+1)$, hence $\dim\Ical^0_{k,m,r}\leq r-2$.Therefore, the closure of its image under the projection onto the first factor is a proper subset of $\PP(\KK[z]^0_{\leq r})\cong\PP^{r-1}$, and the complement of the corresponding affine cone is the desired open subset $U_{k,m,r}$.
\end{proof}

\subsection{Polynomials of bounded degree}

  \begin{thm}
\label{thm:boundedDegreeOrigin}
Let $m,k\geq 1$ and assume
$$
r\geq k \left( k(m+1) + \frac{(m-1)(k-1)}{2} \right) = \frac{k}{2}(3km+k-m+1).
$$
Then, there exists a nonempty Zariski open subset $U_{k,m,r}\subseteq\KK[z]^0_{\leq r}$ such that, for every $\sigma\in U_{k,m,r}$ and every $k$-tuple of nonzero, pairwise distinct polynomials $p_1,\ldots,p_k\in\KK[t]^0_{\leq m}$, the polynomials $\sigma(p_1),\ldots,\sigma(p_k)$ are linearly independent.
\end{thm}

\begin{proof}
Let $\Pcal_{k,\leq m}^{0,+}$ be the open subset of $(\KK[t]^0_{\leq m}\setminus\{0\})^k$ consisting of pairwise distinct $k$-tuples. Note that $\dim\Pcal_{k,\leq m}^{0,+}=km$. For $p\in\Pcal_{k,\leq m}^{0,+}$ and $\lambda\in\KK^{k}, \lambda \neq 0$, restrict the map in \eqref{eq:T-variable-degree} to $T^0_{p,\lambda}:\KK[z]^0_{\leq r}\to\KK[t]^0_{\leq rm}$.

In the following, we use the notation of the proof of Lemma~\ref{lem:variable-degree-rank-bound}. In particular, $D$ is the maximal degree among the $p_i$ with $\lambda_i\neq0$, and we write $p_i(t)=t^Du_i(t^{-1})$ for the $p_i$ of degree $D$. Since $p_i(0)=0$, this time we have that $\deg u_i\leq D-1$. Thus, \eqref{eq:DI-bound} improves to $\nu_0\leq(D-1)\binom{s}{2}\leq(m-1)\binom{k}{2}$. Applying the same recurrence argument to $h=\nu_0+1,\ldots,r$ gives at least $\lfloor(r-\nu_0)/s\rfloor$ linearly independent images. Since $s\leq k$ and $D\leq m$, we obtain
$$
\rank T^0_{p,\lambda}\geq L:=\left\lfloor\frac{r}{k}-\frac{(m-1)(k-1)}{2}\right\rfloor.
$$

We consider the incidence variety $\Ical^{0,\leq}_{k,m,r}$ analogous to the one in \eqref{eq:incidenceBounded}, but with the parameter spaces $\PP(\KK[z]^0_{\leq r})$ and $\Pcal_{k,\leq m}^{0,+}$.
The incidence argument in the proof of Theorem~\ref{thm:variable-degree-main} then yields
$$
\dim\Ical^{0,\leq}_{k,m,r}\leq km+(k-1)+(r-L-1)=r+km+k-L-2.
$$
The numerical assumption on $r$ gives $L\geq k(m+1)$, hence $\dim\Ical^{0,\leq}_{k,m,r}\leq r-2$.Therefore, the closure of the image of $\Ical^{0,\leq}_{k,m,r}$ under the projection onto the first factor is a proper subset of $\PP(\KK[z]^0_{\leq r}) \cong \PP^{r-1}$, and the complement of the corresponding affine cone is the required open subset $U_{k,m,r}$.
\end{proof}

\subsection{Examples for $\sigma$} \label{sec:exAct}

The origin-passing activations $(z^3+z^2+z)^{\ell}$ with $\ell > k(k-2)$ from Example~\ref{ex:sigma_k_ex1} already provide the existence of good activations for the two conjectures of this section.
However, we can easily construct more such examples:

\begin{exmp}
    \label{ex:actOrign1}
    Following the construction of Example~\ref{ex:sigma_k_ex1}, let $k \geq 2$, $\ell > k(k-2)$, $d \geq 2$, and $f=a_1z+\ldots+a_dz^d$ with $a_{d-1}a_d\neq0$. Combining the Newman--Slater theorem~\ref{thm:newman-slater} with Proposition~\ref{prop:twoPolsOrigin} shows that $\sigma_{\ell,f} := f^\ell$ is good, that is, for any pairwise distinct, nonzero, origin-passing polynomials $p_1, \ldots, p_k$, the polynomials $\sigma_{\ell,f}(p_1), \ldots, \sigma_{\ell,f}(p_k)$ are linearly independent. The activation of minimal degree obtained from this construction is $(z^2+z)^{k(k-2)+1}$.
\end{exmp}

\begin{exmp}
    \label{ex:actOrign2} 
    Another example is $\sigma_r := (z+1)^r-1$, which is good in the above sense whenever $r \geq k^2-1$.
    Indeed, let $p_1, \ldots, p_k$ be pairwise distinct, nonzero, origin-passing polynomials, and suppose that $\sum_{i=1}^k \lambda_i \sigma_r(p_i)=0$ for some scalars $\lambda_i$. Setting $q_i:=1+p_i$, this becomes
\begin{align}\label{eq:vahid}
    -\left(\sum_{i=1}^k \lambda_i\right)1^r+\sum_{i=1}^k \lambda_i q_i^r=0.
\end{align}
Since $p_i(0)=0$, all $q_i$ have constant term $1$. Therefore, the polynomials
$1,q_1,\dots,q_k$ are pairwise nonproportional. The Newman--Slater theorem~\ref{thm:newman-slater} shows that the $k+1$ polynomials
$
1^r,q_1^r,\dots,q_k^r
$
are linearly independent, so \eqref{eq:vahid} implies $\lambda_1=\cdots=\lambda_k=0$.

\end{exmp}

\section{Identifiability of Polynomial MLPs} \label{sec:MLP}

In this section, we explain how Conjecture~\ref{conj:weakOrigin} resolves identifiability for deep multilayer perceptrons with sufficiently general polynomial activations. We first formally define such networks. Fix a function $\sigma \colon \RR \rightarrow \RR$, a sequence of $L > 1$ positive integers $d_0, \ldots, d_L$, and, for every $i=1, \ldots, L$, a matrix $W_i \in \RR^{d_i \times d_{i-1}}$. 
\begin{defn} \label{def:mlp}
A \emph{multilayer perceptron} (MLP) with architecture $\mathbf{d} = (d_0, \ldots, d_L)$, activation function $\sigma$ and weights $\mathbf{W} = (W_1, \ldots, W_{L})$ is the function $f_\mathbf{W} \colon \RR^{d_0} \rightarrow \RR^{d_L}$ given by the composition:
\begin{equation}\label{eq:mlpdef}
f_\mathbf{W} = W_{L} \circ \sigma \circ \cdots \circ \sigma \circ W_1,
\end{equation}
where $\sigma$ is applied coordinate-wise.
\end{defn}
We denote by $\mathcal{W} := \bigoplus_{i=1}^L \RR^{d_i \times d_{i-1}}$ the parameter space of an MLP, and by $\varphi \colon \mathcal{W} \ni \mathbf{W} \mapsto f_\mathbf{W}$ its parametrization map. Next, we introduce the notion of a subnetwork. 
\begin{defn}\label{defn:subnet}
Let $0<j<L$ and $1\leq i\leq d_j$. The $i$-th neuron in layer $j$ is called
\begin{itemize}
  \item \emph{inactive} if its incoming or outgoing weights vanish, i.e., $W_j[i,:]=0$ or $W_{j+1}[:,i]=0$;
  \item \emph{redundant} if $W_j[i,:]$ is nonzero but equal to $W_j[k,:]$ for some $1\leq k\leq d_j$ with $k\neq i$.
\end{itemize}
We call $\mathbf{W}$ a \emph{subnetwork} if it contains at least one inactive or redundant neuron. 
\end{defn}
\begin{rmk} \label{rem:subnet}
    If $\mathbf{W}$ is a subnetwork, then $f_\mathbf{W}$ can be represented by an MLP with a smaller architecture. Indeed, if the $i$-th column of $W_{j+1}$ vanishes, it can be removed together with the $i$-th row of $W_j$, obtaining weights for the architecture $(d_0,\ldots,d_{j-1},d_j-1,d_{j+1},\ldots,d_L)$ whose associated function coincides with $f_\mathbf{W}$. Assuming $\sigma(0)=0$, the same holds when the $i$-th row of $W_j$ vanishes. Similarly, when the $i$-th and the $k$-th rows of $W_j$ coincide, one can remove the $i$-th row of $W_j$ and add the $i$-th column of $W_{j+1}$ to its $k$-th column.
\end{rmk}

\begin{defn}\label{defn:identifiability}
The weights $\mathbf{W}\in\mathcal{W}$ are \emph{identifiable} if, for every $\mathbf{V}\in\mathcal{W}$ with $f_\mathbf{W} = f_{\mathbf{V}}$, there exist permutation matrices $P_1 \in \RR^{d_1 \times d_1}, \ldots, P_{L-1} \in \RR^{d_{L-1} \times d_{L-1}}$ such that 
\begin{equation}\label{eq:permiden}
    V_i =P_i W_i P_{i-1}^\top
\end{equation}
for every $i=1,\ldots,L$, where $P_0$ and $P_L$ are identity matrices. 
\end{defn}
In other words, weights are identifiable if all other weights in the same fiber under $\varphi$ are obtained by permuting the neurons in each layer.

\begin{cor}[of Conjecture~\ref{conj:weakOrigin}]\label{thm:deeppoly}
Fix an architecture $\mathbf{d} = (d_0, \ldots, d_L)$.
Let $D:=\max\{d_1,\ldots,d_{L-1}\}$ and, using the notation in Conjecture~\ref{conj:weakOrigin}, let $r \geq R(2D)$. There is a nonempty Zariski open subset $U\subseteq \RR[z]^0_{\leq r}$ with the following property: 
For every MLP of architecture $\mathbf{d}$ with activation $\sigma \in U$,
the weights $\mathbf{W}\in\mathcal{W}$ are identifiable if and only if they are not a subnetwork. 
\end{cor}

\begin{proof}
It is clear from Remark~\ref{rem:subnet} that subnetworks are not identifiable. For the converse, we argue by induction on the depth $L$. For $L=1$, the statement is straightforward. Thus, we assume $L>1$. 
Note that a statement of Conjecture~\ref{conj:weakOrigin} for $2D$ polynomials also applies to any smaller collection with the same degree bound $R(2D)$, since one can complete it to $2D$ nonzero pairwise distinct origin-passing polynomials. 
For every architecture $(e_0, \ldots, e_{L-1})$ satisfying $1 \leq e_i \leq d_i$, the induction hypothesis provides a nonempty Zariski open subset of good activations in $\RR[z]^0_{\leq r}$.
We let $U' \subseteq \RR[z]^0_{\leq r}$ be the intersection of those subsets. 
Moreover, setting $m:=r^{L-2}$ to be the maximum degree of the polynomials output by the network after $L-1$ layers, we consider the nonempty Zariski open subset $U_{2D,r,m}$ from Conjecture~\ref{conj:weakOrigin}.
We then define the nonempty Zariski open subset $U:=U'\cap U_{2D,r,m}$ of $\RR[z]^0_{\leq r}$, and let $\sigma \in U$.

We consider weights $\mathbf{W}\in\mathcal{W}$ for the architecture $\mathbf{d}$ that are not a subnetwork. 
Denote by $p_j$ the output of the 
$j$-th neuron in the penultimate layer of $f_\mathbf{W}$ (before applying $\sigma$).
In other words, $p_j$ is the MLP with architecture $\mathbf{d}' := (d_0, \ldots, d_{L-2},1)$ and weights $\mathbf{W}'_j := (W_1,\ldots,W_{L-2},W_{L-1}[j,:])$.

We claim that the $p_j$ are nonzero polynomials. This follows from the inductive hypothesis as follows:
If the weights $\mathbf{W}'_j$ are not a subnetwork for the architecture $\mathbf{d}'$, then they are identifiable (by the inductive hypothesis) and thus $p_j$ cannot be zero (since otherwise it would be parametrized both by the nonzero weights $\mathbf{W}'_j$ and by the all-zero weights).
However, even though $\mathbf{W}$ is not a subnetwork, the weights $\mathbf{W}'_j$ might be, since some entries of $W_{L-1}[j,:]$ (but not all!) can be zero, resulting in inactive neurons at the $(L-2)$-th layer. As explained in Remark~\ref{rem:subnet}, these neurons can be removed. This operation may generate inactive neurons with vanishing outgoing weights at earlier layers, which can be removed iteratively. In this way, we eventually obtain a smaller MLP with non-subnetwork weights parametrizing $p_j$. 
Now, we can apply the inductive hypothesis and conclude that these weights are identifiable for the smaller MLP. Hence, as above, we see that $p_j$ cannot be zero.

Next, we claim that the $p_j$ are pairwise distinct. If $p_j=p_k$ for some $j\neq k$, consider the MLP with architecture $\mathbf{d}'$ and weights $(W_1,\ldots,W_{L-2},W_{L-1}[j,:]-W_{L-1}[k,:])$, which parametrizes the zero polynomial. Since $W_{L-1}[j,:]\neq W_{L-1}[k,:]$, the same reduction procedure as above yields a smaller MLP with non-subnetwork weights parametrizing zero, again contradicting that these weights must be identifiable by the inductive hypothesis.

Now, consider weights $\mathbf{V}\in\mathcal{W}$ such that $f_\mathbf{W}=f_\mathbf{V}$. Let $q_k$ be the output of the $k$-th neuron in the penultimate layer of $f_\mathbf{V}$ (before applying $\sigma$). Then, for every $i=1,\ldots,d_L$,
\begin{equation}\label{eq:mlpident}
    \sum_{j=1}^{d_{L-1}} W_L[i,j]\,\sigma(p_j)-\sum_{k=1}^{d_{L-1}} V_L[i,k]\,\sigma(q_k)=0.
\end{equation}
Fix $j \in \{1, \ldots, d_{L-1}\}$.
Since $W_L$ has no vanishing columns, there exists $i$ such that $W_L[i,j]\neq0$. For such an $i$, apply Conjecture~\ref{conj:weakOrigin} to \eqref{eq:mlpident}, after removing zero terms and grouping equal polynomials $q_\bullet$ and $p_\bullet$. It follows that $p_j=q_k$ for some $k$. 

Since the $p_j$ are pairwise distinct as $j$ varies, the resulting map $j\mapsto k$ is injective, hence bijective. Let $P_{L-1}$ be the corresponding permutation matrix. The linear independence in \eqref{eq:mlpident}, guaranteed by Conjecture~\ref{conj:weakOrigin}, then gives $V_L=W_LP_{L-1}^\top$.
Moreover, the depth $L-1$ networks with weights $(W_1,\ldots,W_{L-2},P_{L-1}W_{L-1})$ and $(V_1,\ldots,V_{L-1})$ parametrize the same function. The first network is not a subnetwork, so the inductive hypothesis yields permutation matrices $P_1,\ldots,P_{L-2}$ such that $V_i=P_iW_iP_{i-1}^\top$ for $i=1,\ldots,L-2$ (with $P_0=I$) and $V_{L-1}=P_{L-1}W_{L-1}P_{L-2}^\top$. This proves that $\mathbf{W}$ is identifiable.
\end{proof}

\begin{rmk}
    \label{rem:provenIdentifiability}
Even though Conjecture~\ref{conj:weakOrigin} is open in general, its established cases from Section~\ref{sec:originPassing} imply, via the same proof, the following specific identifiability results:
\begin{itemize}
    \item \textit{Width-one networks:} Proposition~\ref{prop:twoPolsOrigin} gives the conclusion of Corollary~\ref{thm:deeppoly} when all hidden widths are one, that is, $d_i=1$ for $1\leq i\leq L-1$.
    \item \textit{Layerwise activations:} Theorem~\ref{thm:boundedDegreeOrigin} gives the analogue of Corollary~\ref{thm:deeppoly} for layer-specific origin-passing activations. More precisely, consider
    $$
    f_\mathbf{W}=W_L\circ\sigma_{L-1}\circ\cdots\circ\sigma_1\circ W_1,
    $$
    where $\deg\sigma_i=r_i$. Set $m_0:=1$ and $m_i:=r_1\cdots r_i$. The outputs from the previous layer entering $\sigma_i$ have degree at most $m_{i-1}$. Hence, it is enough to choose $r_i$ recursively so that
    $$
    r_i \geq d_i (6d_im_{i-1}+2d_i-m_{i-1}+1),
    $$
    and to choose $\sigma_i$ in the nonempty Zariski open subset $U_{2d_i,m_{i-1},r_i}$ given by Theorem~\ref{thm:boundedDegreeOrigin}.
    \item \textit{Shallow networks:} As a particular case, Corollary~\ref{thm:deeppoly} holds unconditionally for $L=2$. Here, $m=1$ and $k=2d_1$, so Theorem~\ref{thm:boundedDegreeOrigin} gives the sufficient bound $r\geq 7d_1^2$.
    \item \textit{Specific choices of activations:} Furthermore, for deep MLPs with any of the specific activations from Section~\ref{sec:exAct} (with $k = 2D$), the weights are also identifiable if and only if they are not a subnetwork.
\end{itemize}
\end{rmk}

\begin{rmk}
    The results in this section concern MLPs without bias vectors. To deal with MLPs that have bias vectors, one generalizes Definition~\ref{def:mlp} by allowing affine linear maps instead of the linear maps $W_i$. As discussed above, the identifiability of MLPs with bias vectors would follow from a strengthening of Conjectures~\ref{conj:strong} and~\ref{conj:weak} that asks for the $k+1$ polynomials $1, \sigma(p_1), \ldots, \sigma(p_k)$ to be linearly independent. 
\end{rmk}

\section{Conclusion}
We have related the identifiability of multilayer perceptrons with generic polynomial activations to a linear-independence problem for polynomial compositions.
We established this linear independence in several special cases, resolving the identifiability question for some MLP architectures, most notably for shallow MLPs and deep MLPs whose activation degrees grow from layer to layer.

For all MLP architectures with a full identifiability characterization, the following remarkable result holds. For every MLP $f_\mathbf{W}$ that is equivariant under some group action, there are weights $\mathbf{V}$ parametrizing the same function (i.e., $f_\mathbf{V}=f_\mathbf{W}$) such that every single layer $x \mapsto \sigma(V_ix)$ is equivariant \cite{shahverdi2026identifiable}.
Hence, for all such architectures, whenever one wishes to implement equivariant MLPs, it is sufficient to implement layerwise-equivariant ones.

Another application of identifiability is model stitching \cite{bansal2021revisiting}, which asks whether independently trained networks have compatible internal representations such that the lower layers of one network can be plugged into the upper layers of the other after only a simple alignment map. 
For MLP architectures with a full identifiability characterization, model stitching is possible since for two networks that are 
sufficiently well-trained so that they represent the same function, a permutation of the neurons in the ``stitching layer'' is the desired alignment map.

Overall, the identifiability of parametrized machine learning models is one of the most classical and fundamental questions in deep learning theory, lying at the heart of many further questions of interest (e.g., equivariant network design, model stitching). 
A full resolution of identifiability for all polynomial MLP architectures would be a foundational contribution to deep learning. Such a resolution could be achieved by proving the conjectures in this article, or possibly by other means. 
At their core, these conjectures ask whether the Newman--Slater theorem can be generalized from monomials to generic polynomials. 
Thus, beyond their consequences for neural-network identifiability, 
the conjectures are of independent interest in commutative algebra and algebraic geometry.

\section*{Acknowledgements}

We are immensely grateful to Vahid Shahverdi for numerous invaluable discussions on this topic. We also thank Jan Draisma and his wonderful group at the University of Bern for stimulating discussions. KK and GLM were supported by the Wallenberg AI, Autonomous Systems and Software Program (WASP) funded by the Knut and Alice Wallenberg Foundation. GLM was supported by the AI4S initiative funded by the Knut and Alice Wallenberg Foundation. 
AM is a member of GNSAGA and was partially supported by the PRIN 2022 project 20223B5S8L, ``\emph{Birational Geometry of Moduli Spaces and Special Varieties}''.
MM is a member of GNSAGA and was partially supported by PRIN 2022 project 2022ZRRL4C ``\emph{Multilinear Algebraic Geometry}''.

\bibliographystyle{plainnat}
\bibliography{main}

\begin{thebibliography}{21}
\providecommand{\natexlab}[1]{#1}
\providecommand{\url}[1]{\texttt{#1}}
\expandafter\ifx\csname urlstyle\endcsname\relax
  \providecommand{\doi}[1]{doi: #1}\else
  \providecommand{\doi}{doi: \begingroup \urlstyle{rm}\Url}\fi

\bibitem[Bansal et~al.(2021)Bansal, Nakkiran, and Barak]{bansal2021revisiting}
Yamini Bansal, Preetum Nakkiran, and Boaz Barak.
\newblock Revisiting model stitching to compare neural representations.
\newblock In \emph{Advances in Neural Information Processing Systems}, volume~34, 2021.

\bibitem[Cr{\u{a}}ciun(2025)]{cruaciun2025linear}
Alexandru Cr{\u{a}}ciun.
\newblock Linear independence of powers for polynomials.
\newblock \emph{arXiv:2507.10163}, 2025.

\bibitem[Fefferman(1994)]{fefferman1994reconstructing}
Charles Fefferman.
\newblock Reconstructing a neural net from its output.
\newblock \emph{Revista Matem{\'a}tica Iberoamericana}, 10\penalty0 (3):\penalty0 507--556, 1994.

\bibitem[Finkel et~al.(2024)Finkel, Rodriguez, Wu, and Yahl]{finkel2024activation}
Bella Finkel, Jose~Israel Rodriguez, Chenxi Wu, and Thomas Yahl.
\newblock Activation thresholds and expressiveness of polynomial neural networks.
\newblock \emph{arXiv:2408.04569}, 2024.

\bibitem[Gegenfurtner et~al.(2026)Gegenfurtner, Grillo, and Mont{\'u}far]{gegenfurtner2026symmetries}
Johanna~Marie Gegenfurtner, Moritz Grillo, and Guido Mont{\'u}far.
\newblock The symmetries of three-layer relu networks.
\newblock \emph{arXiv:2605.18319}, 2026.

\bibitem[Grigsby et~al.(2023)Grigsby, Lindsey, and Rolnick]{grigsby2023hidden}
Elisenda Grigsby, Kathryn Lindsey, and David Rolnick.
\newblock Hidden symmetries of relu networks.
\newblock In \emph{International Conference on Machine Learning}, pages 11734--11760. PMLR, 2023.

\bibitem[Grillo and Mont{\'u}far(2026)]{grillo2026most}
Moritz Grillo and Guido Mont{\'u}far.
\newblock Most relu networks admit identifiable parameters.
\newblock \emph{arXiv:2605.03601}, 2026.

\bibitem[Henry et~al.(2025)Henry, Marchetti, and Kohn]{henry2024geometrylightningselfattentionidentifiability}
Nathan~W. Henry, Giovanni~Luca Marchetti, and Kathlén Kohn.
\newblock Geometry of lightning self-attention: Identifiability and dimension.
\newblock In \emph{International Conference on Learning Representations}, 2025.

\bibitem[Marchetti et~al.(2025)Marchetti, Shahverdi, Mereta, Trager, and Kohn]{marchetti2025invitationneuroalgebraicgeometry}
Giovanni~Luca Marchetti, Vahid Shahverdi, Stefano Mereta, Matthew Trager, and Kathlén Kohn.
\newblock Algebra unveils deep learning -- an invitation to neuroalgebraic geometry.
\newblock In \emph{International Conference on Machine Learning}, 2025.
\newblock URL \url{https://arxiv.org/abs/2501.18915}.

\bibitem[Massarenti and Mella(2025)]{massarenti2025alexander}
Alex Massarenti and Massimiliano Mella.
\newblock The alexander-hirschowitz theorem for neurovarieties.
\newblock \emph{arXiv:2511.19703}, 2025.

\bibitem[Newman and Slater(1979)]{newman1979waring}
DJ~Newman and Morton Slater.
\newblock Waring's problem for the ring of polynomials.
\newblock \emph{Journal of Number Theory}, 11\penalty0 (4):\penalty0 477--487, 1979.

\bibitem[Petzka et~al.(2020)Petzka, Trimmel, and Sminchisescu]{petzka2020notes}
Henning Petzka, Martin Trimmel, and Cristian Sminchisescu.
\newblock Notes on the symmetries of 2-layer relu-networks.
\newblock In \emph{Proceedings of the northern lights deep learning workshop}, volume~1, pages 6--6, 2020.

\bibitem[Phuong and Lampert(2020)]{phuong2020functional}
Mary Phuong and Christoph~H Lampert.
\newblock Functional vs. parametric equivalence of relu networks.
\newblock In \emph{International conference on learning representations}, 2020.

\bibitem[Shahverdi et~al.(2025)Shahverdi, Marchetti, and Kohn]{shahverdi2024geometryoptimizationpolynomialconvolutional}
Vahid Shahverdi, Giovanni~Luca Marchetti, and Kathlén Kohn.
\newblock On the geometry and optimization of polynomial convolutional networks.
\newblock In \emph{Artificial Intelligence and Statistics}, 2025.

\bibitem[Shahverdi et~al.(2026{\natexlab{a}})Shahverdi, Marchetti, B{\"o}kman, and Kohn]{shahverdi2026identifiable}
Vahid Shahverdi, Giovanni~Luca Marchetti, Georg B{\"o}kman, and Kathl{\'e}n Kohn.
\newblock Identifiable equivariant networks are layerwise equivariant.
\newblock \emph{arXiv:2601.21645}, 2026{\natexlab{a}}.

\bibitem[Shahverdi et~al.(2026{\natexlab{b}})Shahverdi, Marchetti, and Kohn]{shahverdi2026learning}
Vahid Shahverdi, Giovanni~Luca Marchetti, and Kathl{\'e}n Kohn.
\newblock Learning on a razor’s edge: Identifiability and singularity of polynomial neural networks.
\newblock In \emph{The Fourteenth International Conference on Learning Representations}, 2026{\natexlab{b}}.

\bibitem[Sussmann(1992)]{sussmann1992uniqueness}
H{\'e}ctor~J Sussmann.
\newblock Uniqueness of the weights for minimal feedforward nets with a given input-output map.
\newblock \emph{Neural networks}, 5\penalty0 (4):\penalty0 589--593, 1992.

\bibitem[Trager et~al.(2020)Trager, Kohn, and Bruna]{trager2019pure}
Matthew Trager, Kathl{\'e}n Kohn, and Joan Bruna.
\newblock Pure and spurious critical points: a geometric study of linear networks.
\newblock In \emph{International Conference on Learning Representations}, 2020.

\bibitem[Usevich et~al.(2025)Usevich, D{\'e}rand, Borsoi, and Clausel]{usevich2025identifiability}
Konstantin Usevich, Clara D{\'e}rand, Ricardo Borsoi, and Marianne Clausel.
\newblock Identifiability of deep polynomial neural networks.
\newblock \emph{arXiv:2506.17093}, 2025.

\bibitem[Vla{\v{c}}i{\'c} and B{\"o}lcskei(2021)]{vlavcic2021affine}
Verner Vla{\v{c}}i{\'c} and Helmut B{\"o}lcskei.
\newblock Affine symmetries and neural network identifiability.
\newblock \emph{Advances in Mathematics}, 376:\penalty0 107485, 2021.

\bibitem[Vla{\v{c}}i{\'c} and B{\"o}lcskei(2022)]{vlavcic2022neural}
Verner Vla{\v{c}}i{\'c} and Helmut B{\"o}lcskei.
\newblock Neural network identifiability for a family of sigmoidal nonlinearities.
\newblock \emph{Constructive Approximation}, 55\penalty0 (1):\penalty0 173--224, 2022.

\end{thebibliography}

\end{document}